\documentclass[12pt]{article}               
\usepackage{geometry}
\usepackage{cite}
\usepackage{multirow}
\usepackage{amsmath}
\usepackage{amssymb}
\usepackage[usenames]{color}
\usepackage{graphicx}          
 \usepackage{amsthm}                              
\usepackage{epsfig}    
 \usepackage{helvet}

\newtheorem{definition}{Definition}
\newtheorem{theorem}{Theorem}

\title{\large Control design for partial stabilization of nonlinear mechanical systems with random disturbances}

\author{Alexander Zuyev\thanks{
Institute of Applied Mathematics and Mechanics, National Academy of Sciences of Ukraine, Slovyansk ({\tt\small  irisna.shurko@gmail.com})
        \newline
$^{**}$Max Planck Institute for Dynamics of Complex Technical Systems, Magdeburg, Germany
({\tt\small zuyev@mpi-magdeburg.mpg.de})\newline
} $^{,**}$ \;and Iryna Vasylieva$^{*}$}
\date{}

\begin{document}

\maketitle
\thispagestyle{empty}

\begin{abstract}
The problem of partial stabilization for nonlinear control systems described by the Ito stochastic differential equations is considered. For these systems, we propose a constructive control design method which leads to establishing the asymptotic stability in probability of the trivial solution of the closed-loop system with respect to a part of state variables.
Mechanical examples are presented to illustrate the efficiency of the obtained controllers.
\end{abstract}

\section{Introduction}
To construct adequate mathematical models that describe the behavior of real dynamic processes and analyze their stability properties, it is necessary to take into account the effects of uncertainties and random disturbances.
The latter leads to the need to study systems of differential equations with random perturbations.
Here, qualitative methods for investigating the asymptotic behavior of solutions of systems of differential equations with random disturbances are useful.
Lyapunov methods for analyzing the stability of stochastic systems have been developed by many authors (see, e.g.,~\cite{Kush_67,Hasm_81} and references therein). In particular, the concept of control Lyapunov functions and Artstein's theorem \cite{Artstein} have been extended to stochastic differential equations in~\cite{Florch:95}. In \cite{Yua:19}, a criterion for stochastic finite-time stability via multiple Lyapunov functions has been obtained.

Partial stabilization problem arises in tasks when only the stability with respect to some variables is needed for a desired performance of the system. This task is also crucial when the system is not stable in the sense of Lyapunov, but asymptotically stable with respect to a part of variables~\cite{RO,CDC2003,UMJ2006,Vor:19}.  Therefore, the problems of partial stability and stabilization of motion are highly important in engineering applications, cf.~\cite{CCA2001,Z2015}. In the paper~\cite{Vor:19}, conditions of partial stability in probability for the Ito stochastic differential equations have been obtained by Lyapunov's direct method. In~\cite{Kao:14},  sufficient conditions for partial stability of stochastic reaction-diffusion systems with Markovian switching have been derived.

In this paper, we consider the problem of stabilization of the Ito-type stochastic differential equations with respect to a part of variables.
Our goal is to propose an efficient control design scheme for the above problem.
To achieve this goal, we present an extension of the universal stabilizing controllers from~\cite{Sontag} to the problem of partial stabilization of stochastic systems in Section~3.
Our main theoretical contribution will be applied to mechanical examples in Sections~4 and~5.

\section{Notations and definitions}

Throughout this paper, let $w(t)\in{\mathbb R}^k$  $(t\ge 0)$ be a standard $k$-dimensional Wiener process defined on a complete probability space $(\Omega,{\mathcal F},P)$, and let $\{{\cal F}_t\}_{t\ge 0}$ be the complete right-continuous filtration generated by $w$.

Consider a control system described by the Ito stochastic differential equations:
\begin{equation}\label{f-1} dx(t) = (f(x) + g(x)u)dt + \sum_{i=1}^k{\sigma_i(x)dw_i(t)},\end{equation}
where $x=(x_1, ..., x_n)^T\in D\subseteq \mathbb{R}^n$ is the state and $u=(u_1, ..., u_k)^T\in U =  \mathbb{R}^k$ is the control.
We assume that $0\in D$, $\sigma_i(0)=0$ for $i=1,...,k$, and
the maps $f:D\to \mathbb R^n$, $g:D\to \mathbb R^{n\times k}$,  $\sigma_i:D\to \mathbb R^n$ satisfy the Lipschitz condition on every bounded domain $X\subset D$.

For a map $h:D\to U$, $h(0)=0$, we introduce the closed-loop system for~\eqref{f-1} with the feedback law $u=h(x)$:
\begin{equation}
dx (t) = (f(x) + g(x)h(x))dt + \sum_{i=1}^k{\sigma_i(x)dw_i(t)}.
\label{closed-loop}
\end{equation}
If $h$ is Lipschitz continuous on every bounded $X\subset D$, then there exists a unique strictly Markov process $x^{\xi ,s}(t)$ which is a solution of~\eqref{closed-loop} under the initial condition $x^{\xi ,s}(s)=\xi$ (see, e.g.,~\cite{Oks_03}).
We relate with the control system~(\ref{f-1}) the operator
\begin{equation*} \mathcal L_{u}=\sum_{i=1}^n (f(x)+g(x)u)_{i}\frac{\partial }{\partial x_{i}}+\frac{1}{2}\sum_{i,j=1}^{n}c_{ij}(x)\frac{\partial^{2}}{\partial x_{i}\partial x_{j}}, \end{equation*}$$[c_{ij}(x)]=\sigma(x)\sigma^{T}(x).$$

In the sequel, we will study stability of the trivial solution of~(\ref{closed-loop}) with respect to the variables $x_1, x_2,...,x_m.$ Denote these variables as $y=(y_1,...,y_m)^T \in \mathbb{R}^m$ and the rest as $z=(z_1,...,z_p)^T\in \mathbb{R}^p,$ $m+p=n$, then $x=(y^T,z^T)^T,$ $x_{0}=(y_0^T,z_0^T)^T$, and $||x||=(x_1^2+...+x_n^2)^{1/2}=(||y||^2+||z||^2)^{1/2}.$

We assume also that the solutions of~(\ref{f-1}) are $z-$extendable in a closed domain
$D=\mathcal D_H$, where
$$\mathcal D_H = \{x\in {\mathbb R}^n\,:\,||y(t)||\le H,\; z\in \mathbb R^p\},\; H=const > 0.
$$
It means that if $x(t) \in \mathcal D_H$ is a maximal solution of system  (\ref{f-1}) on $t \in (\tau _{1}, \tau_ {2})$ with some admissible control $u \in L^{\infty}(\tau _{1}, \tau_ {2})$, then either $||y(t)|| \rightarrow H $ as $t \rightarrow \tau_2$ almost surely or $\tau_2 = \infty.$
This kind of $z$-extendability assumption is natural in the problems of partial stability~\cite{RO}; it is usually satisfied for well-posed mathematical models in physics whose trajectories do not blow up in finite time with bounded control.

Let us introduce the standard class of comparison functions $\mathcal K,$ whose elements are continuous strictly increasing functions $\alpha:\mathbb{R}^+\to\mathbb{R}^+$ such that $\alpha(0)=0.$
We will extend the concept of a control Lyapunov function~\cite{Artstein,Florch:95,Sontag,Z:2000} to the problem of partial stabilization of stochastic systems as follows.

\begin{definition}
A function $V\in C^2({\cal D}_H;{\mathbb R})$ is called a $y$-stochastic control Lyapunov function  ($y$-SCLF) for system (\ref{f-1}), if there exist $\alpha, \beta_1, \beta_2 \in \mathcal {K}$ such that
\begin{equation*}
\beta_1(||y||)\le V(x)\le \beta_2(||y||),\end{equation*} \begin{equation*} \inf_{u\in U} {\mathcal L}_uV(x) \le -\alpha (||y||), \end{equation*} for all $x \in \mathcal D_H.$\end{definition}

Throughout the text, $B(x;\delta)$ denotes the $\delta$-neighborhood of a point $x\in{\mathbb R}^n$.

\begin{definition}
 A function $V\in C^2({\cal D}_H;{\mathbb R})$ satisfies the small control property with respect to $y$  if, for any $\epsilon >0$ and any $x_0 \in M=\{x|y=0\}$, there exists a $\delta >0 $ such that
 $$x \in B( x_0; \delta) \Rightarrow \inf_{||u||< \epsilon} \mathcal L_u V(x) \le -\alpha (||y||).$$
 \end{definition}

\begin{definition}\cite{Vor:19,Sharov,We,Ign}
The solution $x=0$ of system (\ref{closed-loop}) is called $y$-stable in probability if, for all $
s\ge 0, \varepsilon >0, \gamma >0$, there exists a  $\delta >0$ such that $\xi\in B(0;\delta)$ implies
\begin{equation*} P\{\sup_{t\geq s}
||y^{\xi,s}(t)||>\varepsilon\}<\gamma.\end{equation*}\end{definition}

\begin{definition}\cite{We,Ign}
The solution $x= 0$ of system (\ref{closed-loop}) is called asymptotically $y$-stable in probability if it is $y$-stable in probability and
\begin{equation*}P\{\lim_{t\rightarrow \infty}{||y^{\xi,s}(t)||=0\}}=1\end{equation*} for all $\xi \in B(0;\Delta)$ with some constant $\Delta  > 0$.\end{definition}

\section{Main result}
The following result generalizes the constructive proof of Artstein's theorem~\cite{Sontag} for the problem of partial stabilization of stochastic systems.
\begin{theorem}
Let $V\in C^2({\mathcal D}_H;\mathbb R)$ be a $y$-SCLF satisfying the small control property. Then there exists a continuous feedback law $h:{\mathcal D}_H \to \mathbb{R}^k$, $h(0,z)=0$, such that the trivial solution of the corresponding closed-loop system~\eqref{closed-loop} with $u=h(x)$ is  $y$-asymptotically stable in probability. The feedback law $h(x)$ is given as follows:
\footnotesize\begin{equation}\label{f-7}h_i(x)=\begin{cases} 0,\qquad\qquad \qquad\qquad\qquad b=0,\\
-\frac{b_i}{\|b\|^2}(a+(a^2+\|b\|^4)^\frac{1}{2}),\;\;  b\ne 0, 2(a^2+\|b\|^4)^\frac{1}{2} \ge \alpha(\|y\|),\\
-\frac{b_i}{2\|b\|^2}(2a+\alpha(\|y\|)), \qquad \mbox{otherwise},\end{cases}\end{equation}\normalsize
where$$ a(x)=\sum_{i=1}^n f_i(x)\frac{\partial V(x)}{\partial x_i}+\frac{1}{2} \sum_{i,j=1}^{n}c_{ij}(x)\frac{\partial^{2}V(x)}{\partial x_{i}\partial x_{j}},$$
\begin{equation}
b_i(x)=\sum_{j=1}^{n}g_{ij}(x)\frac{\partial V(x)}{\partial x_{j}},\;
b(x)=(b_1(x),...,b_k(x)).
\label{ab_formula}
\end{equation}
\end{theorem}

\begin{proof}
The proof of continuity of $h(x)$ in~\eqref{f-7} goes along the same lines as the proof of Theorem~4 in~\cite{Z:2000}.

Let us evaluate the operator  $\mathcal L_{u}V$ for~(\ref{f-1}) using the feedback law $u=h(x)$:
\begin{equation*}{\mathcal L}_{h}V=\begin{cases} a(x),\qquad\qquad \qquad\qquad b=0,\\
-(a^2(x)+\|b(x)\|^4)^\frac{1}{2}, \; b\ne 0, 2(a^2+\|b\|^4)^\frac{1}{2} \ge \alpha(\|y\|),\\
-\frac{1}{2}\alpha(\|y\|), \qquad \mbox{otherwise.}\end{cases}\end{equation*}
As $V(x)$ is a $y$-stochastic control Lyapunov function, the following inequality holds:
\begin{equation*}{\mathcal L}_{h}V \le -\frac{1}{2}\alpha(||y||)\; \text{for all}\; x\in {\cal D}_H.\end{equation*}

Using Gr\"onwall's inequality, we have:
$$E||y^{\xi,s}(t)||^2\le k_1(t-s)e^{\int_s^t{k_2E||y^{\xi,s}(p)||^2}dp}\le N_1 e^{N_2\delta^2},$$
where  $E$ is the expectation in the probability measure $P_{\xi,s},$ $y^{\xi,s}(t)$ is the $y$-component of the solution $x^{\xi,s}(t)$ of~(\ref{closed-loop}) with the initial data $x^{\xi,s}(s)=\xi.$

Putting $\delta=ln(\frac{\epsilon_2 \epsilon_1^2}{N_1})^{\frac{1}{2N_2}},$ we get
$$P\{sup_{t\ge t_0} ||y(t)||>\epsilon_1\}\le\frac{E||y(t)||^2}{\epsilon_1^2}<\epsilon_2.$$

Let $\tau_{\varepsilon}=\inf\{t:\|y^{\xi,s}(t)\|>\varepsilon\},$ $\tau_{\varepsilon}(t)=\min(\tau_{\varepsilon},t).$

From Dynkin's lemma ~\cite{Oks_03}, it follows that
$$E V(x^{\xi,s}(\tau_{\varepsilon}(t))) - V(\xi) = E \int_{s}^{\tau_\varepsilon(t)} {\mathcal L}_h V(x^{\xi,s}(u))du.$$
 Since ${\mathcal L}_hV (x) \le -\frac{1}{2}\alpha (||y||) $, we will get \begin{equation}\label{f-10}E V(x^{\xi,s}(\tau_{\varepsilon}(t)))\leq V(\xi),\quad t\ge s.\end{equation}

 The above inequality can be rewritten as  $$\int_{\tau_\varepsilon<t} \alpha_{1}(\|y^{\xi,s}(\tau_{\varepsilon})\|)P_{\xi,s}(d\omega)+$$$$+\int_{\tau_\varepsilon\geq t}\alpha_{1}(\|y^{\xi,s}(t)\|)P_{\xi,s}(d\omega)\leq V(\xi).$$
Hence,  $$\alpha_{1}(\varepsilon)P_{\xi,s}\{\tau_{\varepsilon}<t\}\leq
V(\xi).$$

From the last equality, due to the continuity of the function $V(x)$ and the equality  $V(0)=0$, it follows that $$\lim_{\xi\rightarrow
0}P_{\xi,s}\{\tau_\varepsilon<t\}=0.$$ So, the   equilibrium $x=0$ of system (\ref{closed-loop}) is $y$-stable in probability.

  From (\ref{f-10}) it follows that the random process $V(x^{\xi,s}(\tau_{\varepsilon}(t)))$ is a nonnegative supermartingale, and there exists the limit \begin{equation}\label{f-11}\lim_{t\rightarrow \infty}V(x^{\xi,s}(\tau_{\varepsilon}(t)))=\eta\end{equation} with probability~1.

From the set of sample trajectories of the process $x^{\xi,s}(t)$ we take the subset $B$ of sample trajectories such that for any $x_{i}^{\xi,s}(t)$ $(i=1,..., n)$ the following  equality holds: $\tau_{\varepsilon}(t)=t,
t\in\mathbb R^{+}$. Then it follow from the above assumptions that
\begin{equation}\label{f-12}\lim_{\xi_{y} \rightarrow 0}P_{\xi,s}\{B\}=1,\end{equation}
where $\xi^{T}=(\xi_{y}^{T},\xi_{z}^{T}).$

From (\ref{f-11}) and (\ref{f-12}),  we have \begin{equation}\label{f-13}\lim_{t\rightarrow \infty}V(x^{\xi,s}(\tau_{\varepsilon}(t)))=\lim_{t\rightarrow \infty}V(x^{\xi,s}(t))=\eta.\end{equation}

Note that $V(x)$ is a $y$-stochastic control Lyapunov function, so for all trajectories from the set $B$, except a set of probability $0,$ the following property holds: $$ \lim_{t\rightarrow \infty}||y^{\xi,s}(t)||=0.$$

From the assumption of $z-$extendability of solutions and (\ref{f-13}), we obtain $\eta = 0.$

So,  $ \lim_{t\rightarrow \infty}||y^{\xi,s}(t)||=0.$ From this property it follows that the zero solution of the closed-loop system (\ref{closed-loop}) is $y$-asymptotically stable in probability.
\end{proof}

\section{Inverted pendulum with a moving mass}

To illustrate possible applications of Theorem~3.1, we consider a mechanical system consisting of
an inverted pendulum (carrier body) and a point mass $m$ moving in the direction perpendicular to the axis of symmetry of the carrier body  (Fig.~1).
It is assumed that the mass $m$ is suspended by a spring with the stiffness coefficient $\varkappa$.

\begin{figure}[h!]
\center {
\includegraphics[scale=0.27]{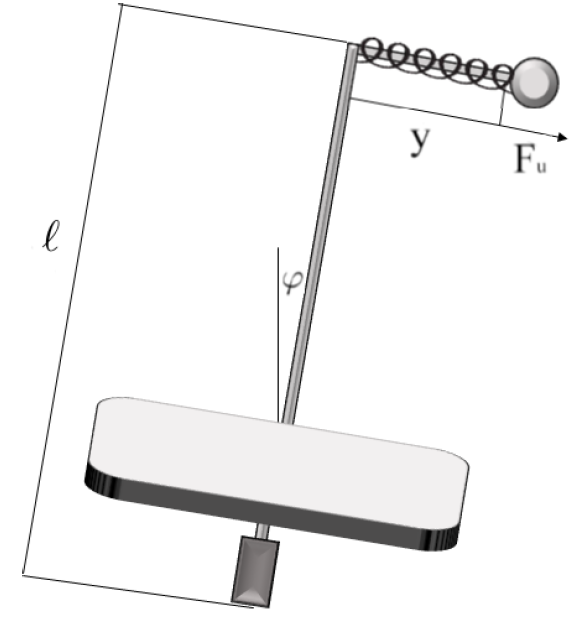}
}
\caption{Inverted pendulum with a moving mass.}

\end{figure}

We will use the following notations: $M$ is the mass of the carrier body, $\varphi$ is the angle between the axis of symmetry of the carrier
and the vertical,
$y$ is the displacement of the point mass,
 and $\ell$ is the distance between the fixed point and the suspension of the mass $m$.

Let us first derive  the equations of motion of this mechanical systems by using the Lagrangian formalism.
The  kinetic energy of the system is $$T=\left(\frac{I}{2}+\frac{m(\ell^2+y^2)}{2}\right){\dot{\varphi}}^2+\frac{m}{2}\dot{y}^2+m\ell\dot{\varphi}\dot{y},$$
where $I$ is the moment of inertia of the carrier body with respect to its fixed point.
The potential energy is $$U=\frac{M\ell g}{2}\cos\varphi+\frac{\varkappa}{2}y^2+mg(\ell \cos\varphi-y\sin\varphi).$$

Then the Lagrangian of the considered system takes the form
$$L=T-U=\left(\frac{I}{2}+\frac{m({\ell}^2+y^2)}{2}\right){\dot{\varphi}}^2+\frac{m}{2}\dot{y}^2+m{\ell}\dot{\varphi}\dot{y}-$$
$$-\frac{\varkappa}{2}y^2-\left(\frac{M}{2}+m\right){\ell}g\cos\varphi+mgy\sin\varphi.$$

We now apply Lagrange's equations in the form
\begin{equation*}
\begin{array}{lcl}
\frac{d}{dt}\left(\frac{\partial L}{\partial \dot{\varphi}}\right)-\frac{\partial L}{\partial \varphi}=0,\\
\frac{d}{dt}\left(\frac{\partial L}{\partial \dot{y}}\right)-\frac{\partial L}{\partial y}=F_u,
\end{array}
\end{equation*}
where $F_u$ is the control force applied to the mass $m$.

This leads to the following equations of motion:
\begin{equation*}
\begin{array}{lcl}
\ddot{\varphi}=\frac{1}{I+my^2}(-2my\dot{\varphi}\dot{y}-m{\ell}y{\dot{\varphi}}^2+\varkappa {\ell}y+\frac{M{\ell}g}{2}\sin \varphi+\\ \quad\quad\quad\quad\quad\quad\quad\quad +mgy\cos \varphi - {\ell}F_u),\\
\ddot{y}=\frac{{\ell}}{I+my^2}(2my\dot{\varphi}\dot{y}+m{\ell}y{\dot{\varphi}}^2-\varkappa {\ell}y-\frac{M{\ell}g}{2}\sin \varphi+\\ + {\ell}F_u - m g y\cos \varphi)+\frac{I}{I+my^2}(y{\dot{\varphi}}^2-\frac{\varkappa y}{m}+g \sin \varphi)+\\+\frac{1}{I+my^2}((\frac{I}{m}+y^2)F_u+my^3{\dot{\varphi}}^2+y^3 \varkappa + y^2 m g \sin \varphi).
\end{array}
\end{equation*}

By replacing $$v=\frac{1}{I+my^2}(-2my\dot{\varphi}\dot{y}-m{\ell}y{\dot{\varphi}}^2+\varkappa {\ell}y+\frac{M{\ell}g}{2}\sin \varphi +mgy\cos \varphi -{\ell}F_u),$$ we obtain the following equations with respect to the new control $v$:
\begin{equation*}
\begin{array}{lcl}
\ddot{\varphi}=v,\\
\ddot{y}=-({\ell} +\frac{I+my^2}{m \ell} )v+\frac{1}{I+my^2}(2y^3 \varkappa +\frac{2m+M}{2m}(\frac{I}{m}+\\ \quad\quad\quad\quad\quad\quad\quad\quad + y^2)g \sin \varphi)-\frac{2y\dot{y}\dot{\varphi}}{\ell} +\frac{g y \cos \varphi}{\ell}.
\end{array}
\end{equation*}

Let us rewrite the above equations of motion in the form $\dot{x} = f(x)+g(x)v$,

\begin{equation}\label{f-24}
x=\begin{bmatrix}x_1 \\ x_2 \\ x_3 \\ x_4 \end{bmatrix}=\begin{bmatrix}
       \varphi\\
       y\\
       \dot{\varphi}\\
       \dot{y}
     \end{bmatrix},
     f(x)=
     \begin{bmatrix}
     x_3\\
       x_4 \\
       0\\
       q(x)
     \end{bmatrix},
     g(x)=
     \begin{bmatrix}
       0\\
       0\\
       1\\
       -{\ell}-\frac{I+my^2}{m \ell}
     \end{bmatrix},\end{equation}
    $$ q(x)=\frac{1}{I+m x_2^2}\left (2x_2^3 \varkappa +\frac{2m+M}{2m}\left (\frac{I}{m}+ x_2^2 \right )g \sin x_1\right )-\frac{2 x_2 x_3 x_4 }{\ell} +\frac{g x_2 \cos x_1}{\ell}.$$

It is easy to see that system~\eqref{f-24} admits the equilibrium $x=0$ with $v=0$ (upper equilibrium).
We will consider the stabilization of the upper equilibrium of the carrier body in the sense of partial stabilization problem with respect to the variables $(x_1,x_3)$ by applying control to the point mass.

    To take into account random effects, we substitute the stochastic input  $v=u+\lambda x_3 \dot{w}(t)$ formally into system~\eqref{f-24},
    where $w(t)$ is a standard one-dimensional Wiener process.
    As a result, we obtain the following system of stochastic differential equations:
\begin{equation}\label{f-25}
\begin{array}{lcl}
d x_1= x_3 dt,\\
d x_2= x_4 dt,\\
d x_3= u dt+\lambda x_3 d{w}(t),\\
d x_4=\left((-{\ell}-\frac{I+my^2}{m \ell})u+q(x)\right)dt-({\ell}+\frac{I+my^2}{m \ell})\lambda x_3d{w}(t),
\end{array}
\end{equation}
where $u$ is treated as the control.

Since our goal is to steer the variables $\varphi$ and $\dot{\varphi}$ (i.e. $x_1$ and $x_3$) to zero, we propose the following quadratic Lyapunov function candidate:
$$2V(x)=(k_1^2+k_2^2+k_2)x_3^2+2k_1 x_1 x_3+(k_2+1)^2x_1^2,$$
where $k_1$ and $k_2$ are positive constants.

Let us define the functions $a(x)$ and $b(x)$ according to~\eqref{ab_formula}:
$$ a(x)=\sum_{i=1}^4 f_i(x)\frac{\partial V(x)}{\partial x_i}+\frac{1}{2} \sum_{i,j=1}^{4}c_{ij}(x)\frac{\partial^{2}V(x)}{\partial x_{i}\partial x_{j}}=
$$
$$
=(k_1 x_3+(k_2+1)x_1)x_3+(k_2^2+k_1^2+k_2){\lambda}^2 x_3^2,
$$
$$b(x) = (k_2^2+k_1^2+k_2)x_3+k_1 x_1.$$

According to Theorem~3.1, we propose the feedback control law for system (\ref{f-25}) in the form~\eqref{f-7} with $\alpha(\|y\|)= \gamma \|y\|^2$, $\|y\|^2= x_1^2+ x_3^2$, $\gamma>0$.
So, the equilibrium $x=0$ of the corresponding closed-loop system (\ref{f-25}), (\ref{f-7}) is asymptotically stable in probability with respect to $(x_1, x_3)$ by Theorem~3.1.
Simulation results for the closed-loop system (\ref{f-25}), (\ref{f-7}) with $k_1=2,$ $k_2=1$  are presented in Fig.~2.
These simulations have been performed in Maple by using the $ItoProcess(\cdot)$ function.

\begin{figure}[h!]
\center {
\includegraphics[scale=0.3]{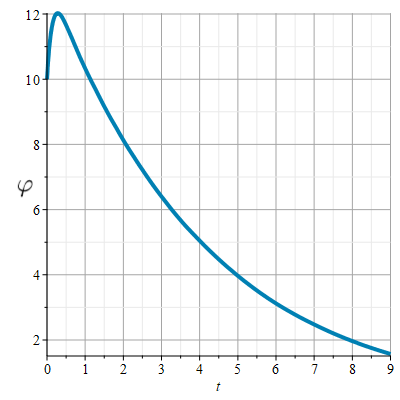}
\includegraphics[scale=0.3]{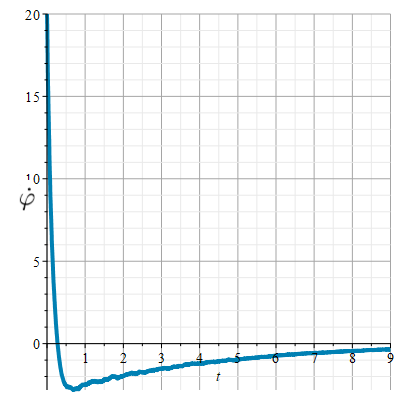}

}
\caption{Components $x_1$ and $x_3$ of a sample path of the closed-loop system~(\ref{f-25}), (\ref{f-7}).}

\end{figure}

\section{Stabilization of a three-wheeled trolley by a stochastic feedback law}

\begin{figure}[h!]
\center {
\includegraphics[scale=0.4]{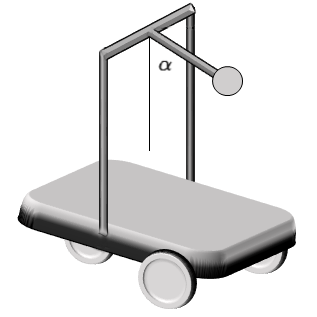}
}
\caption{Three-wheeled trolley.}

\end{figure}

Consider a mathematical model of the three-wheeled trolley whose position is determined by three coordinates: $(x_1, x_2 )$ are coordinates of the midpoint between the steering wheels, and $x_3 $ is the angle between the axis of symmetry of the trolley and the $x_1$-axis, cf.~\cite{K:2006}.
 A cylindrical hinge whose axis is perpendicular to the axis of symmetry of the trolley is mounted above the point $(x_1,x_2)$   (Fig.~3). A weightless and inextensible rod can rotate in this hinge, and a point mass is attached to the other end of the rod. We denote the angle between the vertical axis and the rod by $\alpha$. The motion of the trolley is described by the rolling without slipping conditions:
\begin{equation}\label{f-14}
\begin{array}{lcl}
dx_{1}=(u_1+u_2)\cos x_3dt,\\
dx_{2}=(u_1+u_2)\sin x_3dt,\\
dx_{3}=(u_1-u_2)dt,
\end{array}
\end{equation}
where the vector $u=(u_1,u_2)^T \in \mathbb{R}^2$ is treated as the control.

Following~\cite{K:2006}, we also write Lagrange's equation with respect to the angle $\alpha$:
\begin{equation}\label{f-15}
\ddot{\alpha}-(\dot{x_1}\cos x_3 + \dot{x_2}\sin x_3 + \dot{x_3}\sin \alpha)\dot{x_3} \cos \alpha =-\sin \alpha.
\end{equation}

Note that the considered model belongs to the class of nonholonomic systems which, as it is well-known, cannot be stabilized in a neighborhood of the equilibrium position by a deterministic continuous state feedback law (see, e.g.,~\cite{Bloch}).
In the sequel, we will study the stabilization problem with respect to a part of variables in the stochastic sense.

Let us denote the relative angular velocity of the rod by $ \omega = \dot{\alpha}$ and perform the following change of variables in~(\ref{f-14}), (\ref{f-15}):
\begin{equation*}
\begin{array}{lcl}
z_{1}:= x_3,\\
z_{2}:=x_{1}\cos x_3+x_2\sin x_3,\\
z_{3}:=x_{1}\sin x_3-x_2\cos x_3,\\
z_{4}:=\alpha,\\
z_{5}:=\omega,\\
\nu_{1}:=u_1-u_2,\\
\nu_{2}:=(u_1+u_2)-(u_1-u_2)z_{3}.
\end{array}
\end{equation*}

Then the equations of motion take the form:
\begin{equation}\label{f-17}
\begin{array}{lcl}
\dot{z_{1}}=\nu_1,\\
\dot{z_{2}}=\nu_2,\\
\dot{z_{3}}=\nu_1 z_2,\\
\dot{z_{4}}=z_5,\\
\dot{z_{5}}=(\nu_2+\nu_1 z_3+\nu_1 \sin z_4)\nu_1 \cos z_4-\sin z_4.
\end{array}
\end{equation}

We randomize system~(\ref{f-17})  by designing the control inputs
$$\begin{array}{lcl}
\nu_1=v_1,\\
\nu_2=v_2+\lambda z_{2}\dot w(t),
\end{array}$$
where $\dot w(t)$ is treated formally as the derivative of a standard one-dimensional Wiener process $w(t)$.
Then we rewrite the stochastic control system as follows:
\begin{equation}\label{f-18}
\begin{array}{lcl}
d{z_{1}}=v_1dt,\\
d{z_{2}}=v_2dt+\lambda z_{2}d w(t),\\
d{z_{3}}=v_1 z_2dt,\\
d{z_{4}}=z_5dt,\\
d{z_{5}}=\left((v_2+v_1 z_3+v_1 \sin z_4)v_1 \cos z_4-\sin z_4\right)dt+\\\quad  \quad +\lambda z_{2}v_1 \cos z_4d w(t).
\end{array}
\end{equation}

We consider the partial stabilization problem for  system (\ref{f-18})  with respect to the variables $z_1, z_2, z_3.$

To design  stabilizing controls $v_1,v_2,$ we take a control Lyapunov function candidate of the following form~\cite{N:2013}:
$$V(z)=2z_3-\frac{1}{2}(z_{1}^2+z_{2}^2)(1+z_{3}^2)+2\left(\frac{|z_{1}^2+z_{2}^2|}{2}\right)^{1+\frac{z_{3}^2}{2}}.$$

Then we define the functions $a(z),b_1(z),b_2(z)$ according to~\eqref{ab_formula}:
$$a(z) =\frac{1}{2} \sum_{i,j=1}^{5}c_{ij}(z)\frac{\partial^{2}V(z)}{\partial z_{i}\partial z_{j}}= \frac{1}{2}{\lambda}^2 z_{2}^2\frac{\partial^{2}V(z)}{{\partial z_{2}}^2},$$
$$b_1(z) = -z_1(z_{3}^2+1)+\frac{4\left(\frac{|z_{1}^2+z_{2}^2|}{2}\right)^{1+\frac{z_{3}^2}{2}}(1+\frac{z_{3}^2}{2})z_1}{|z_{1}^2+z_{2}^2|}+z_2\left(2-(z_{1}^2+z_{2}^2)z_3+\right.$$
$$+ z_2\left(2-(z_{1}^2+z_{2}^2)z_3+2\left(\frac{|z_{1}^2+z_{2}^2|}{2}\right)^{1+\frac{z_{3}^2}{2}}z_3 \ln\left(\frac{|z_{1}^2+z_{2}^2|}{2}\right)\right),$$
$$b_2(z) = -z_2(z_{3}^2+1)+\frac{4\left(\frac{|z_{1}^2+z_{2}^2|}{2}\right)^{1+\frac{z_{3}^2}{2}}(1+\frac{z_{3}^2}{2})z_2}{|z_{1}^2+z_{2}^2|},$$
$$b(z)=(b_1(z),b_2(z)).$$

Thus, the conditions of  Theorem~3.1 are satisfied with the above choice of $a(x)$, $b(x)$, and $\alpha(\|y\|) =\gamma\|y\|^2$, $\|y\|^2= z_1^2+z_2^2+z_3^2$, $\gamma>0$.
Numerical simulation results for system (\ref{f-18}) with the feedback law (\ref{f-7}) are presented in Figs.~4-5.

\begin{figure}[h!]
\center {
\includegraphics[scale=0.3]{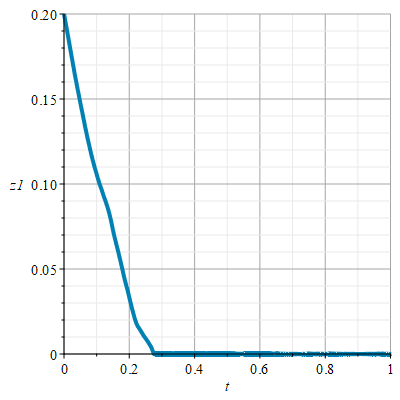}
\includegraphics[scale=0.3]{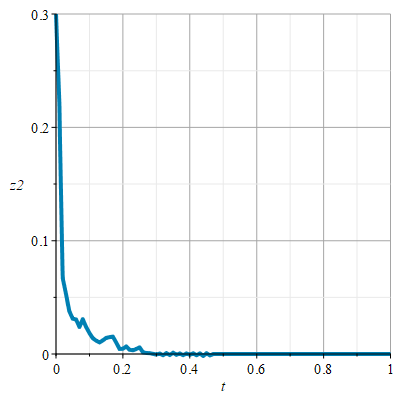}
\includegraphics[scale=0.3]{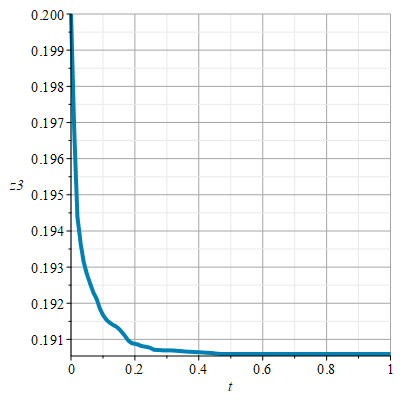}
}
\caption{Components $z_1,z_2,z_3$ of a sample path of the closed-loop system~(\ref{f-18}), (\ref{f-7}).}

\end{figure}

\begin{figure}[h!]
\center {
\includegraphics[scale=0.3]{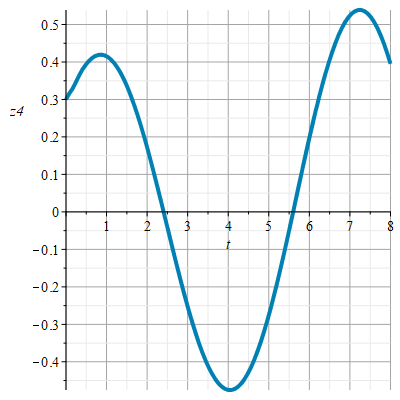}
\includegraphics[scale=0.3]{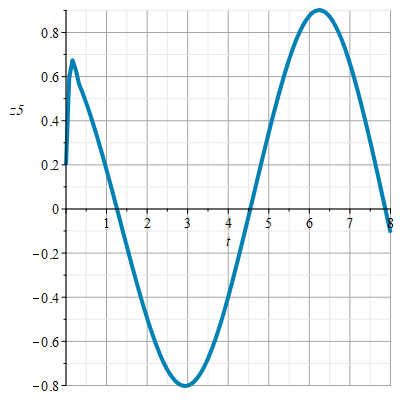}
}
\caption{Components $z_4,z_5$ of a sample path of the closed-loop system~(\ref{f-18}), (\ref{f-7}).}

\end{figure}

\section{Conclusion}

 A constructive proof of Artstein's theorem has been extended to the problem of partial stabilization of the Ito stochastic differential equations.
 This construction allows effective computing of stabilizing feedback controls if a control Lyapunov function in the sense of Definitions 2.1-2.2 is known.
 The control design scheme of Theorem~3.1 is shown to be applicable to nonlinear systems with stochastic effects that describe the dynamics of an inverted pendulum with a moving masses and a three-wheeled trolley with an additional degree of freedom.
The simulation results, presented in Figs. 2 and 4-5, illustrate the required behavior of sampled paths of the corresponding closed-loop systems.

\end{document}